\documentclass[12pt]{amsart}

\usepackage{amsmath}
\usepackage{amsfonts}
\usepackage{amssymb}
\usepackage{amsthm}
\usepackage{hyperref}
\usepackage{tikz}
\usepackage{enumerate}
\usepackage{cleveref}
\usepackage{mathrsfs}
\usepackage[top=0.9in, bottom=1.15in, left=1.1in, right=1.1in]{geometry}
 \usepackage{hyperref}
\hypersetup{colorlinks=true,linkcolor=blue,citecolor=magenta}
\newtheorem{theorem}{Theorem}%[section]
\newtheorem{lemma}[theorem]{Lemma}
\newtheorem{corollary}[theorem]{Corollary}
\newtheorem*{remark}{Remark}
\Crefname{conjecture}{Conjecture}{Conjectures}

\theoremstyle{plain}

\theoremstyle{plain}

\author{Matthew Just and Robert Schneider}

\address{Department of Mathematics\newline
University of Georgia\newline
Athens, Georgia 30602, U.S.A.}
\email{justmatt@uga.edu}

\address{Department of Mathematics\newline
University of Georgia\newline
Athens, Georgia 30602, U.S.A.}
\email{robert.schneider@uga.edu}

\title{Partition Eisenstein series and semi-modular forms}

\begin{document}

\begin{abstract}
We identify a class of ``semi-modular'' forms invariant on special subgroups of $GL_2(\mathbb Z)$, which includes classical modular forms together with complementary classes of functions that are also nice in a specific sense. We define an Eisenstein-like series summed over integer partitions, and use it to construct families of semi-modular forms.
\end{abstract} %UPDATED ABSTRACT SLIGHTLY SINCE SENDING TO MATT

\maketitle

\section{Introduction and statement of results}
\subsection{Semi-modular forms}
In a landmark 2000 paper \cite{B-O}, Bloch and Okounkov introduced an operator from statistical physics, the $q$-bracket, under the action of which %has the interpretation of the expected value of certain functions over statistical-mechanical systems whose equilibrium states are indexed by integer partitions. Bloch-Okounkov proved conditions for types of $q$-series that, under the action of the $q$-bracket,
%transforms
certain partition-theoretic $q$-series %$q$-series
transform to quasi-modular forms, a class of functions that includes classical modular forms. This work was expanded on by Zagier \cite{Zagier} and subsequent authors, e.g. \cite{Padic, Schneider_arithmetic, Schneider_JTP, JWVI}. In recent work \cite{BOW}, Bringmann-Ono-Wagner produce examples of modular forms, quantum modular forms, and harmonic Maass forms via relations to classical Eisenstein series and properties of $t$-hooks from the theory of integer partitions, again by applying the $q$-bracket.

These works display an intriguing theme: patterns and symmetries within the set $\mathcal P$ of partitions, give rise to modularity properties. % --- noting the underlying nature of the many connections between $q$-series and modular forms is, at present, enigmatic. % to contemporary researchers.
In this paper, we apply similar ideas to answer a theoretical question we pose regarding the existence of classes of special functions: we construct a class of Eisenstein series summed over partitions --- and dependent on symmetries within the set $\mathcal P$ --- to produce first examples of what we call ``semi-modular forms'', which in short are functions of a complex variable enjoying one of the two canonical invariances of modular forms, as well as a strong complementary invariance property.

Let us recall the canonical generators of the {\it general linear group} $GL_{2}(\mathbb Z)$, viz.
\begin{equation*}
T=\begin{pmatrix}
1 & 1 \\
0 & 1
\end{pmatrix},
\  \  \  \
U=\begin{pmatrix}
0 & 1 \\
1 & 0
\end{pmatrix},
\  \  \  \
V=\begin{pmatrix}
1 & 0 \\
0 & -1
\end{pmatrix}.
\end{equation*}
An important subgroup of $GL_{2}(\mathbb Z)$ is the {\it modular group} $PSL_{2}(\mathbb Z)$, which is well known to be generated by the ``translation'' matrix $T$ together with the ``inversion'' matrix
\begin{equation*}
S=\begin{pmatrix}
0 & -1 \\
1 & 0
\end{pmatrix}
\in GL_{2}(\mathbb Z).\end{equation*}
Functions on $z \in \mathbb H$ (upper half-plane) invariant under $\left<S,T\right>=PSL_{2}(\mathbb Z)$ up to a simple multiplier in $z$ are {\it modular forms}, a class of functions central to modern number theory. % and physics.

Of course, the canonical generators $T,U,V$ for  $GL_2(\mathbb Z)$ are not unique. Then given that the matrices $S,T$ induce modular forms, a natural question to ask is: can one find examples of a ``nice''
complementary matrix $R\in GL_2(\mathbb Z)$ such that
\begin{equation*}
GL_2(\mathbb Z)=\left<R, S, T\right>,
\end{equation*}
that come with ``nice'' complementary functions invariant on $\left<R, S\right>$ and/or $\left<R, T\right>$, as well as modular forms invariant on $\left<S,T\right>$?
Such families of functions would be, in a sense, ``half-modular'' or {\it semi-modular} as we will denote them, in that they are invariant with respect to two of the three generators $R,S,T$ of $GL_2(\mathbb Z)$, including at least
one of the generators of the modular group.
Take for instance the following matrix $R\in GL_2(\mathbb Z)$:
\begin{equation}
R:=\begin{pmatrix}
-1 & 0 \\
0 & 1
\end{pmatrix}.
\end{equation}
Even functions are the forms invariant under $R$, a very nice class of functions. In recent work \cite{DuncanMcGady}, the matrix $R$ is found by Duncan-McGady to play a natural role in their theory of half-integral weight modular forms defined on the double half-plane $\mathbb C \backslash \mathbb R$. %\footnote{For further reading on modularity in the double half-plane $\mathbb C \backslash \mathbb R$, see e.g. \cite{examples}.}
%, motivated by ideas from physics.

Indeed, observing that one can rewrite the generators $U,V$ of $GL_2(\mathbb Z)$ as
\begin{equation*}
U=RS,\  \  \  \  V=RS^2,
\end{equation*}
%{\bf (need to check above equations})
%SPTS_^{-1} = V
%PTS = U
%A,B,C
%A^2 = B^2 = (BA)^4 = (CABA)^2 = (CAB)^3 = (CB)^2=1
%
%A = RST  = USST
%B = R      = US
%C = S      = S
%
%where
%
%R = [-1,0;0,1]
%S = [1,1;0,1]
%T = [0,-1;1,0]
%U= [-1,1;0,1]
 then one can alternatively view %the general linear group as
%\begin{equation}
$GL_2(\mathbb Z)=\left< R,S,T\right>$, as desired.
%\end{equation}
Note that the subgroup $\left<R,T\right>$ induces invariance in even periodic functions of period 1, e.g. the function
\begin{equation*}
f(z)=\cos(2\pi z),\  \  \  \  z\in \mathbb C,
\end{equation*}
is in this class. %Such even periodic functions are semi-modular, by the above definition, in that they are invariant under one generator $T$ of $PSL_2(\mathbb Z)$ as well as the complementary generator $M$ of $GL_2(\mathbb Z)$.
From this perspective, modular forms invariant on $\left<S,T\right>$ and even periodic functions of period 1 invariant on $\left<R,T\right>$ are members of the larger class of {semi-modular forms} defined above, functions invariant on two out of three of the generators $R,S,T$.\footnote{While modular forms are also semi-modular, we note that the identification of a third, complementary generator of $GL_2(\mathbb Z)$ such as $R$ is implicit in our definition of semi-modularity.} 

\subsection{Partition-theoretic Eisenstein series}
Giving further consideration to the generators $R,S,T$ of $GL_2(\mathbb Z)$, one wonders about the other, more obscure-feeling sibling class of functions invariant on $\left<R,S\right>$ that should complete the family  $\{\left<R,S\right>,\left<R,T\right>,\left<S,T\right> \}$ of ``semi-modular forms'', complementing modular forms and even periodic functions of period 1. % invariant on these subgroups of $GL_2(\mathbb Z)$.
This type of semi-modular family does not appear to be widely studied in the literature; we seek explicit examples. Here we construct semi-modular forms invariant under $\left<R,S\right>$ to complete the family, by fusing classical Eisenstein series with ideas from partition theory.

Recall for $k>2, \tau\in \mathbb H$, the weight $k$ {\it Eisenstein series} (a double summation) which is the prototype of an integer-weight holomorphic modular form invariant on $PSL_2(\mathbb Z)$:
\begin{equation*}%\label{Eisenstein}
G_{k}(\tau)=\sum_{\substack{a,b\in\mathbb Z \\ (a,b)\neq (0,0)}} (a\tau+b)^{-k}.
\end{equation*}
If the weight $k>2$ is odd, then % that %every term cancels its negative, and
$G_k(\tau)=0$. Taking $k \mapsto 2k$, then for $k>1$, the function $G_{2k}:\mathbb H \to \mathbb C$ satisfies the defining properties of a weight $2k$ modular form: %defining properties of an {\it integer-weight modular form}:
\begin{enumerate}[(i)]
\item $G_{2k}(-\frac{1}{\tau})=\tau^{2k}G_{2k}(\tau)$\   (weighted invariance under inversion matrix $S$),
\item $G_{2k}(\tau+1)=G_{2k}(\tau)$ \  (invariance under translation matrix $T$).
\end{enumerate}

We mimic this classical function using ideas from partition theory, to construct semi-modular partition-theoretic Eisenstein series invariant on $\left<R,S\right>$. Now, for $N\geq 1$ one immediately sees that the truncated Eisenstein series
\begin{equation}\label{truncated}
\sum_{\substack{|a|\leq N,|b|\leq N  \\ (a,b)\neq (0,0)}} (az+b)^{-2k}
\end{equation}
still respects (i) above, and is an even function of $z\not\in \mathbb R$, %; moreover, the full Eisenstein series $G_{2k}(z), k>1,$  is clearly invariant on {\it all three} generators $R,S,T$ for $z\not\in \mathbb R$,
an observation expanded on significantly in \cite{DuncanMcGady}. % \in \mathbb C\backslash \mathbb R$.
Thus the double sum \eqref{truncated} is semi-modular with respect to $\left< R,S\right>$%, as is $G_{2k}(z)$ itself
; as we remark explicitly below, this {\it ad hoc} example has a %is subsumed by the
partition-theoretic interpretation. % we provide.

Recall $\mathcal P$ is the set of integer partitions, including the empty partition $\emptyset$. As we detail explicitly in Section \ref{Sect2} below, we define a {\it Ferrers-Young lattice} in the complex plane, %$\mathscr F(\lambda, \tau)$
which is a four-fold symmetric version of the classical Ferrers-Young diagram\footnote{Ferrers diagrams use dots to illustrate partitions; Young diagrams use unit squares in identical arrangements.} of an integer partition $\lambda\in \mathcal P$, with the vertices (dots) of the Ferrers diagram %\footnote{Or lower right corners of the squares comprising the equivalent Young diagram.}
plotted according to a natural rule (see Section \ref{Sect2}) in the lattice $\left<z,1\right> \subset \mathbb C\backslash \mathbb R$ generated by the fundamental pair $z$ and $1$, for fixed $z\in \mathbb C, z\not\in \mathbb R$. %We note  $(0,0)\not\in \mathscr F(\lambda, \tau)\subset \left<\tau,1\right> $.
%, \tau\in \mathbb H$.

Summing over points $\omega=az+b \in \mathscr F(\lambda, z)$ consisting of vertices of the Ferrers-Young lattice for partition $\lambda$, for $k\in \mathbb Z$ we define an auxiliary {\it single-partition Eisenstein series} %for a sinlge partition% of both $\lambda, \tau$: % and for $\lambda \neq \emptyset$, %$f_m(\lambda, \tau)$ whose properties  %:=\{  a +bz : 1\leq |a| \leq k \text{ and } 1\leq |b| \leq \lambda_{|a|} \}$, which essentially is the well-known Ferrers-Young diagram of a partition plotted into the complex plane as a function of $z \in \mathbb H$.  (For more about modular forms, see e.g. \cite{Apostol, Ono_web}.)
\begin{equation}\label{Eisenstein2}
f_{k}(\lambda, z)\  :=\  \frac{1}{4}\sum_{\omega \in \mathscr F(\lambda,z)} \omega^{-k}\  =\  \frac{1}{4}\sum_{\  az+b\in  \mathscr F(\lambda,z)}(az+b)^{-k},
\end{equation}
with $f_k(\emptyset, z):=0$. (The factor ${1}/{4}$ compensates for the four-fold symmetry of $\mathscr{F}(\lambda,z)$.) As we show in Section \ref{proofs}, if $k=0$ then $f_0(\lambda, z)=|\lambda|$, the {\it size} (sum of parts) of partition $\lambda$; thus $f_k(\lambda, z)$ generalizes $|\lambda|$. %This finite series also has nice properties resembling $G_{m}$ as a function of $\tau$; these properties lead to semi-modular behavior

\begin{remark}
In this combinatorial setting, the truncated Eisenstein series \eqref{truncated} represents \begin{flalign*}\sum_{\substack{|a|\leq N,|b|\leq N  \\ (a,b)\neq (0,0)}} (az+b)^{-2k}\  =\  4\cdot f_{2k}\left((N)^{N}, z\right)+(2+z^{2k}+z^{-2k})\sum_{1\leq m\leq N}m^{-2k}\end{flalign*} for $z\not\in \mathbb R$, where $(N)^{N}\in \mathcal P$ is the partition of size $N^2$ consisting of $N$ copies of $N$; thus
\begin{flalign*}  f_{2k}\left((N)^{N}, z\right)\  \sim\  \frac{1}{4}\left[G_{2k}(z)-(2+z^{2k}+z^{-2k})\zeta(2k)\right]\end{flalign*} %The full Eisenstein series $G_{2k}(z), k>1,$ picks up invariance on $T$ in the limit as $N\to \infty$.
as $N\to\infty$ for $k>1$. We note the factor $2+z^{2k}+z^{-2k}$ is also invariant on $\left<R,S\right>$.%; moreover, as $N\to \infty$,
%; we note the second term on the right-hand side is clearly invariant on $\left< R,S\right>$. Thus as $N\to \infty$,
%\begin{equation}  f_{2k}\left((N)^{N}, z\right)\  \sim
%\  \frac{1}{4}\left[G_{2k}(z)+(2+z^{2k}+z^{-2k})\zeta(2k)\right].
%\end{equation}
\end{remark}

Then analytic properties of $f_k(\lambda, z)$ yield semi-modular behavior of the following finite {\it partition Eisenstein series} summed over partitions of $n\geq 1$ for $z\not\in \mathbb R$: % (a double summation by \eqref{Eisenstein2}):
\begin{equation}\label{partEisenstein}
g_{k}(n, z):=\sum_{\lambda \vdash n} f_k(\lambda, z),
\end{equation}
where ``$\lambda \vdash n$''  means $\lambda$ is a partition of $n$, with $g_k(0,z):=0$. We note $g_0(n, z)=n\cdot p(n)$. %, the partition function.
%for $m=0$ we have $f_0(\lambda, \tau)=|\lambda|$, the size (sum of parts) of $\lambda$. Moreover, $f_m$ itself resembles $G_{m}$ as a function of $\tau$; for $m$ odd we have that $f_m(\lambda, \tau)=0$, and we also have $f_{2m}(\lambda, -1/\tau)=\tau^{2m}f_{2m}(\overline{\lambda},\tau)$ where $\overline{\tau}$ is the partition conjugate of $\lambda$.

\begin{theorem}\label{Thm1}
Let $k\in \mathbb Z, n\geq 0, z\not\in \mathbb R$. If $k$ is an odd integer then $g_k(n,z)=0.$ For even weights, $g_{2k}(n,z)$ satisfies the following properties: % defining properties of a weight-$2m$ modular form: %defining properties of an {\it integer-weight modular form}:
\begin{enumerate}[(i)]
\item $g_{2k}(n,-\frac{1}{z})=z^{2k}g_{2k}(n,z)$\   (weighted invariance under $S$),
\item $g_{2k}(n,-z)=g_{2k}(n,z)$ \  (invariance under $R$).
\end{enumerate}
\end{theorem}

By Theorem \ref{Thm1}, then,  $g_{k}(n,z)$ is semi-modular in $z$ over $\left< R, S\right>$; it will serve as a building block for this class of semi-modular forms.
In the next section we see that the identity (i) above has the following combinatorial interpretation: {\it taking $z \mapsto -1/z$ produces partition conjugation in certain ``dual'' Ferrers-Young lattices}.% Indeed, noting $g_{2k}(n,z)$ generalizes the partition size $|\lambda|$, then (i) generalizes the invariance of $|\lambda|$ under conjugation.

 %Immediately, we obtain another representative.

Let us now consider the following two-variable generating function: % for $g_k(n,z)$:
\begin{flalign}\label{partEisenstein2}
\mathscr{G}_{k}(z)\  =\  \mathscr{G}_{k}(z,q)\  :&=  \sum_{n\geq 1} g_k(n, z)q^n\\ \nonumber &=  \sum_{\lambda \in \mathcal P}f_k(\lambda, z)q^{|\lambda|},
\end{flalign}
valid for $z\not\in \mathbb R, |q|<1$. Then $\mathscr{G}_{k}(z)$ inherits semi-modularity in $z$ from its coefficients.

\begin{corollary}\label{Cor1}
Let $k\in \mathbb Z, z\not\in \mathbb R$. If $k$ is an odd integer then $\mathscr{G}_k(z)=0.$ For even weights, $\mathscr{G}_{2k}(z)$ satisfies the following properties: % defining properties of a weight-$2m$ modular form: %defining properties of an {\it integer-weight modular form}:
\begin{enumerate}[(i)]
\item $\mathscr{G}_{2k}(-\frac{1}{z})=z^{2k}\mathscr{G}_{2k}(z)$\   (weighted invariance under $S$),
\item $\mathscr{G}_{2k}(-z)=\mathscr{G}_{2k}(z)$ \  (invariance under $R$).
\end{enumerate}
\end{corollary}

%This infinite partition Eisenstein series $\mathscr{G}_m(z)$ yields a semi-modular form over $\left<R,S\right>$.

%\begin{remark}
Again, the transformation $z\mapsto -1/z$ appears in connection with partition conjugation. As an infinite series, $\mathscr{G}_k(z)=\mathscr{G}_k(z,q)$ appears as something of a close cousin to functions in the realm of modular forms. For instance, the simplest case $k=0$ gives
\begin{equation*}
\mathscr{G}_{0}(z)\  =\  \sum_{n\geq 1} n\cdot p(n)q^n%=q \frac{d}{dq}(q;q)_{\infty}^{-1}
\  =\  \prod_{m\geq 1}(1-q^m)^{-1}\cdot\sum_{n\geq 1}\sigma_1(n)q^n, \end{equation*}
which connects to well-known identities writing $G_{2k}(\tau)$ in terms of $\sum_{n\geq 1}\sigma_{2k-1}(n)q^n$ when $q:=e^{2\pi i \tau},\tau\in \mathbb H$ (see e.g. \cite{Apostol}), linking % between % (see e.g. \cite{Apostol}), directly relating
$ \mathscr{G}_{0}(z)=\mathscr{G}_{0}(z,q)$ to the quasi-modular form $G_2(\tau)$.

\begin{remark}
One wonders more generally if $\prod_{m\geq 1}(1-q^m)\cdot \mathscr{G}_{2k}(z,q)$ is worthy of deeper study, noting by \eqref{partEisenstein} it represents the $q$-bracket of $f_{2k}(\lambda,z)$ in the sense of Bloch-Okounkov.
\end{remark}

%Based on these analogies, natural questions to ask of a family of semimodular functions is: are they finitely generated like modular forms?

%:
%%\end{remark}
%\begin{equation}
%\prod_{k\geq 1}(1-q^k) \cdot \mathscr{G}_{0}(z)= \frac{1}{24}- \frac{G_2(\tau)}{8\pi^2}. \end{equation}

\section{Proofs of Theorem \ref{Thm1} and Corollary \ref{Cor1}}\label{Sect2whole}

\subsection{Ferrers-Young diagrams in the complex plane}\label{Sect2}

Let $\lambda=(\lambda_1,\lambda_2,\ldots,\lambda_r), \lambda_1 \geq \lambda_2 \geq \dots \lambda_r\geq 1$, denote a nonempty partition, with $\emptyset \in \mathcal P$ denoting the empty partition; we call the number of parts $r\geq 0$ the {\it length} of the partition. We write $\lambda \vdash n$ to indicate that $\lambda$ is a partition of $n$, noting that $\emptyset\vdash 0$. We recall the classical {\it Ferrers-Young diagram} of a partition as well as {\it partition conjugation} (swapping rows and columns of the Ferrers-Young diagram); see e.g. \cite{Andrews}. For instance, the partition $\lambda=(3,2,2,1)$ and its conjugate $\overline{\lambda}=(4,3,1)$ have the following Ferrers-Young diagrams, respectively:

\vspace{0.5cm}
%\begin{center}
\hfill 
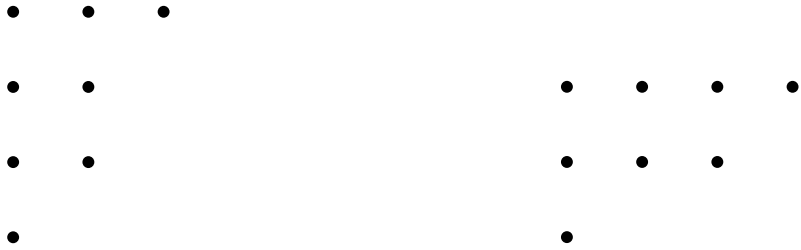
\begin{tikzpicture}[inner sep=0pt,thick,
    dot/.style={fill=black,circle,minimum size=4.5pt}]
\node[dot] (a) at (0,0) {};
\node[dot] (a) at (1,1) {};
\node[dot] (a) at (0,1) {};
\node[dot] (a) at (0,2) {};
\node[dot] (a) at (1,2) {};
\node[dot] (a) at (0,3) {};
\node[dot] (a) at (1,3) {};
%\node[dot] (a) at (0,4) {};
%\node[dot] (a) at (1,4) {};
%\node[dot] (a) at (2,4) {};
%\node[dot] (a) at (3,4) {};
\node[dot] (a) at (2,3) {};
%\draw[-] (-0.1,2.5)--(1.5,2.5);
%\draw[-] (1.5,2.5)--(1.5,4.1);
%\draw[dashed] (-0.1,1.5)--(1.5,1.5);
%\draw[dashed] (1.5,1.5)--(1.5,2.5);
\end{tikzpicture}
%\end{center}
\hfill \hfill \begin{tikzpicture}[inner sep=0pt,thick,
    dot/.style={fill=black,circle,minimum size=4.5pt}]
%\node[dot] (a) at (0,0) {};
%\node[dot] (a) at (1,1) {};
%\node[dot] (a) at (0,1) {};
\node[dot] (a) at (0,2) {};
%\node[dot] (a) at (1,2) {};
\node[dot] (a) at (0,3) {};
\node[dot] (a) at (1,3) {};
\node[dot] (a) at (0,4) {};
\node[dot] (a) at (1,4) {};
\node[dot] (a) at (2,4) {};
\node[dot] (a) at (3,4) {};
\node[dot] (a) at (2,3) {};
%\draw[-] (-0.1,2.5)--(1.5,2.5);
%\draw[-] (1.5,2.5)--(1.5,4.1);
%\draw[dashed] (-0.1,1.5)--(1.5,1.5);
%\draw[dashed] (1.5,1.5)--(1.5,2.5);
\end{tikzpicture}\hfill
 \vspace{0.5cm}

 %of size $n$, viz. $\sum \lambda_k = |\lambda|=n$.
 For any point $z=x+iy \in \mathbb{C}\backslash \mathbb R$ %where \[\mathbb{H} = \{ z\in \mathbb{C} : \Im(z)>0 \}\] is the upper half plane,
 we define a four-fold {\it Ferrers-Young lattice} representing partition $\lambda$, to be the set of points
 \begin{equation}\mathscr F(\lambda,z): = \{  az +b : 1\leq |b| \leq r \text{ and } 1\leq |a| \leq \lambda_{|b|} \},\end{equation}
 where $r$ is the length of $\lambda$. In words, plot the Ferrers-Young diagram on the vertices in the corner of the first quadrant of the lattice $\left<z,1\right>$ such that the top edge of the diagram is parallel to the line $tz, \  t\in \mathbb R$, the left-hand edge is parallel to the real number line, and neither edge falls on its respective border; likewise, plot the diagram into the other three quadrants with top and left edges still parallel to the line $tz$ and real axis, respectively.

For an illustration, set $z=1+i$ and take $\lambda = (3,2,2,1)$ as above. Then the Ferrers-Young lattices $\mathscr F(\lambda,z)$ (left) and $\mathscr F(\overline{\lambda},z)$ (right) are both illustrated, respectively, here: %in Figure \ref{4foldex}.

 \vspace{0.5cm}

%\begin{figure}[h]
    %\centering
    \hfill 
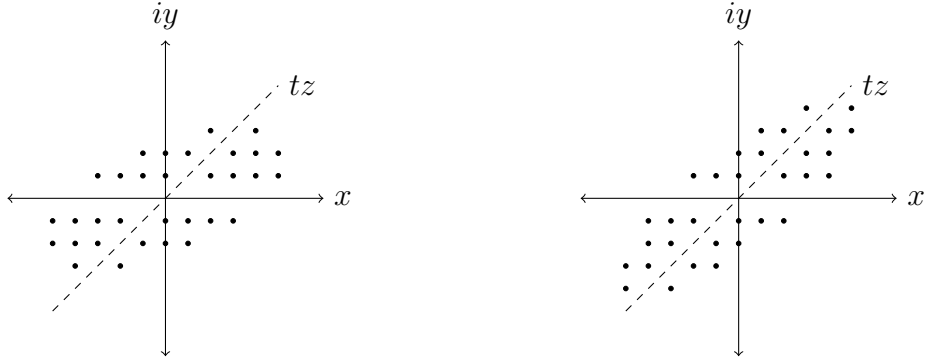
\begin{tikzpicture}[scale=.3]
            \draw[<->] (-7,0)--(7,0)node[anchor=west]{$x$};
            \draw[<->] (0,-7)--(0,7)node[anchor=south]{$iy$};
            \draw[dashed] (-5,-5)--(5,5)node[anchor=west]{$tz$};
            \draw[fill=black] ({1+1},1) circle (.1);
            \draw[fill=black] ({1+2},2) circle (.1);
            \draw[fill=black] ({1+3},3) circle (.1);
            \draw[fill=black] ({2+1},1) circle (.1);
            \draw[fill=black] ({2-1},-1) circle (.1);
            \draw[fill=black] ({1-1},-1) circle (.1);
            \draw[fill=black] ({1-2},-2) circle (.1);
            \draw[fill=black] ({1-3},-3) circle (.1);
            \draw[fill=black] ({-2+1},1) circle (.1);
            \draw[fill=black] ({-1+1},1) circle (.1);
            \draw[fill=black] ({-1+2},2) circle (.1);
            \draw[fill=black] ({-1+3},3) circle (.1);
            \draw[fill=black] ({-2-1},-1) circle (.1);
            \draw[fill=black] ({-1-1},-1) circle (.1);
            \draw[fill=black] ({-1-2},-2) circle (.1);
            \draw[fill=black] ({-1-3},-3) circle (.1);

            \draw[fill=black] (4,1) circle (.1);
            \draw[fill=black] (4,2) circle (.1);
            \draw[fill=black] (5,2) circle (.1);
            \draw[fill=black] (5,1) circle (.1);

            \draw[fill=black] (-2,1) circle (.1);
            \draw[fill=black] (0,2) circle (.1);
            \draw[fill=black] (-1,2) circle (.1);
            \draw[fill=black] (-3,1) circle (.1);

            \draw[fill=black] (-4,-1) circle (.1);
            \draw[fill=black] (-4,-2) circle (.1);
            \draw[fill=black] (-5,-2) circle (.1);
            \draw[fill=black] (-5,-1) circle (.1);

            \draw[fill=black] (2,-1) circle (.1);
            \draw[fill=black] (0,-2) circle (.1);
            \draw[fill=black] (1,-2) circle (.1);
            \draw[fill=black] (3,-1) circle (.1);
            \begin{scope}[xshift=10in]
            \draw[<->] (-7,0)--(7,0)node[anchor=west]{$x$};
            \draw[<->] (0,-7)--(0,7)node[anchor=south]{$iy$};
            \draw[dashed] (-5,-5)--(5,5)node[anchor=west]{$tz$};
            \draw[fill=black] ({1+1},1) circle (.1);
            \draw[fill=black] ({1+2},2) circle (.1);
            \draw[fill=black] ({2+1},1) circle (.1);
            \draw[fill=black] ({3+1},1) circle (.1);
            \draw[fill=black] ({1-1},-1) circle (.1);
            \draw[fill=black] ({1-2},-2) circle (.1);
            \draw[fill=black] ({2-1},-1) circle (.1);
            \draw[fill=black] ({3-1},-1) circle (.1);
            \draw[fill=black] ({-1+1},1) circle (.1);
            \draw[fill=black] ({-1+2},2) circle (.1);
            \draw[fill=black] ({-2+1},1) circle (.1);
            \draw[fill=black] ({-3+1},1) circle (.1);
            \draw[fill=black] ({-1-1},-1) circle (.1);
            \draw[fill=black] ({-1-2},-2) circle (.1);
            \draw[fill=black] ({-2-1},-1) circle (.1);
            \draw[fill=black] ({-3-1},-1) circle (.1);

            \draw[fill=black] (4,3) circle (.1);
            \draw[fill=black] (5,3) circle (.1);
            \draw[fill=black] (4,2) circle (.1);
            \draw[fill=black] (5,4) circle (.1);

            \draw[fill=black] (2,3) circle (.1);
            \draw[fill=black] (3,4) circle (.1);
            \draw[fill=black] (0,2) circle (.1);
            \draw[fill=black] (1,3) circle (.1);

            \draw[fill=black] (-2,-3) circle (.1);
            \draw[fill=black] (-1,-3) circle (.1);
            \draw[fill=black] (0,-2) circle (.1);
            \draw[fill=black] (-3,-4) circle (.1);

            \draw[fill=black] (-4,-3) circle (.1);
            \draw[fill=black] (-5,-4) circle (.1);
            \draw[fill=black] (-4,-2) circle (.1);
            \draw[fill=black] (-5,-3) circle (.1);
            \end{scope}
        \end{tikzpicture}\hfill\hfill
   % \caption{Ferrers-Young lattices for $\lambda=(3,2,2,1)$ (left) and $\overline{\lambda}=(4,3,1)$ (right) in the complex plane.}% Here we use $\tau=1+i$.}
    %\label{4foldex}
 \vspace{0.5cm}
%\end{figure}

The single-partition Eisenstein series $f_k(\lambda,z)$ defined in \eqref{Eisenstein2} is the sum over the vertices $\omega$ of the lattice $\mathscr F(\lambda,z)$. We note that conjugation looks like ``skewed'' rotation within the lattice. Nice properties of $f_k(\lambda,z)$ resulting from symmetries between conjugate partitions together with easy complex-analytic properties of Ferrers-Young lattices, will yield the semi-modularity of $g_{k}(n,z)$ and $\mathscr G_{k}(z)$ in $z\in \mathbb C\backslash \mathbb R$.

%Now for each integer $k$ define \[f_k(\lambda,z):=\frac{1}{4}\sum_{z\in \mathscr F(\lambda,z)} z^{-k}.\]
%{\bf Let's multiply by factor of 1/4 so steries gives size exactly at k=0?}
\subsection{Proofs}\label{proofs}
The main results follow almost immediately from the following lemma.% \ref{Lemma}.
\begin{lemma}\label{Lemma}
 For $k \in \mathbb Z, z \in \mathbb C\backslash\mathbb R$, the function $f_{k}(\lambda,z)$ has the following properties:
    \begin{enumerate}[(i)]
        \item If $k=0$ then $f_{k}(\lambda,z)=|\lambda|$ (thus $f_k$ represents a generalization of partition size).
        \item If $k$ is odd then $f_k(\lambda,z )=0$.
        \item If $k$ is even then $f_k(\lambda,z)$ is an even function of $z$.
        \item Moreover, we have \[z^{2k}f_{2k}(\lambda,z)=f_{2k}\left(\overline{\lambda},-\frac{1}{z}\right).\]
    \end{enumerate}
\end{lemma}

\begin{proof}[Proof of Lemma \ref{Lemma}]
    To establish these properties we start by writing the function $f_k(\lambda,z)$ as defined in \eqref{Eisenstein2} in the following way. Let
    \begin{align*}
        f^{(1)}_k(\lambda,z) &=\frac{1}{4}\sum_{1\leq b\leq r} \sum_{1\leq a\leq \lambda_b} (az+b)^{-k}, \\
        \nonumber f^{(2)}_k(\lambda,z) &=\frac{1}{4}\sum_{1\leq b\leq r} \sum_{1\leq a\leq \lambda_b} (az-b)^{-k}, \\
        \nonumber f^{(3)}_k(\lambda,z) &=\frac{1}{4}\sum_{1\leq b\leq r} \sum_{1\leq a\leq \lambda_b} (-az-b)^{-k}, \\
        \nonumber f^{(4)}_k(\lambda,z) &=\frac{1}{4}\sum_{1\leq b\leq r} \sum_{1\leq a\leq \lambda_b} (-az+b)^{-k},
    \end{align*}
viz. the sums over the four individual quadrants of $\mathscr F(\lambda,z)$, so that
    \begin{align}\label{summands}
    f_k(\lambda,z)&=f^{(1)}_k(\lambda,z)+f^{(2)}_k(\lambda,z)+f^{(3)}_k(\lambda,z)+f^{(4)}_k(\lambda,z)\\
    \nonumber &= \frac{1}{4}\sum_{1\leq b\leq r} \sum_{1\leq a\leq \lambda_b} \left[(az+b)^{-k}+(az-b)^{-k}+(-az-b)^{-k}+(-az+b)^{-k} \right].
    \end{align}
 The claimed properties are easily deduced from the summands on the right-hand side of this expression. If $k=0$, then clearly \begin{equation*} f_0(\lambda,z) = \sum_{1\leq b \leq r}\sum_{1\leq a\leq \lambda_b} 1 = |\lambda|,\end{equation*}
which establishes (i). If $k$ is odd, positive and negative summands cancel: \begin{equation*}(az+b)^{-k}+(az-b)^{-k}-(az+b)^{-k}-(az-b)^{-k}=0,\end{equation*}
    %We now adopt the convention of writing $g_{2k}(\lambda,z)$. Now i
which gives (ii). If we replace $z$ with $-z$, the summands become \begin{equation*} (a(-z)+b)^{-2k}+(a(-z)-b)^{-2k}+(-a(-z)-b)^{-2k}+(-a(-z)+b)^{-2k},\end{equation*}
which are identically the summands on the right side of \eqref{summands}, giving (iii).
%\end{proof}
%
%    We now prove a Lemma that establishes a property similar to a property satisfied by modular forms. Recall that for a partition $\lambda$ the \textit{conjugate} partition $\overline{\lambda}$ is the partition obtained by reflecting the Ferrers-Young diagram across the line $y=-x$. For example, if $\lambda = (3,1)$ then $\overline{\lambda} = (2,1,1)$.
%
%    \begin{lemma}
%        For a partition $\lambda$ and $z\in \mathbb{H}$ we have \[z^{2k}g_{2k}(\lambda,z)=g_{2k}(\overline{\lambda},-1/z).\]
%    \end{lemma}
%
%    \begin{proof}

To prove (iv), note that for even weights $2k\neq 0$, one can write \begin{equation*} f_{2k}(\lambda,z) = \frac{1}{2}\cdot \sum_{1\leq b\leq r}\sum_{1\leq a \leq \lambda_b} \left[(az+b)^{-2k} + (az-b)^{-2k} \right].\end{equation*} Much as with classical Eisenstein series, factoring out $(-z)^{-2k}$ yields
        \begin{align*}
            f_{2k}(\lambda,z)&=(-z)^{-2k}\cdot \frac{1}{2} \cdot \sum_{1\leq b\leq r}\sum_{1\leq a \leq \lambda_b} \left[ (b(-1/z) -a)^{-2k} + (b(-1/z) +a)^{-2k})  \right] \\
          \nonumber  &= (-z)^{-2k}\cdot \frac{1}{2} \cdot \sum_{1\leq a\leq s}\sum_{1\leq b \leq \overline{\lambda}_a} \left[ (a(-1/z)+b)^{-2k} + (a(-1/z)-b)^{-2k})  \right]\\
           \nonumber  &=z^{-2k} f_{2k}(\overline{\lambda},-1/z),
        \end{align*}
        since interchanging coefficients $a,b$ and order of summation in the second equality is equivalent to summing over lattice $\mathscr F (\overline{\lambda},-1/z)$ for the conjugate partition $\overline{\lambda}=(\overline{\lambda}_1, \overline{\lambda}_2, \dots, \overline{\lambda}_s)$.\end{proof}

\begin{remark} An interesting feature of this proof is the explicit interdependence of partition conjugation $\lambda \mapsto \overline{\lambda}$ and complex variable inversion $z\mapsto -1/z$, yielding a duality between $\mathscr F(\lambda,z)$ and $\mathscr F(\overline{\lambda},-1/z)$.\end{remark}

%        We can now state one of our main results. We write \[g_{2k}(n,z) = \sum_{|\lambda|=n} g_{2k}(\lambda,z).\]
%
%        \begin{theorem}
%            For $|q|<1$, $z\in \mathbb{H}$, and any nonnegative integer $k$ the series \[F(q,z)=\sum_{n\geq1}g_{2k}(n,z) q^n\] converges. Furthermore, we have that $F(q,-1/z) = z^{2k}F(q,z)$.
%        \end{theorem}

        \begin{proof}[Proof of Theorem \ref{Thm1}]
We note the finite double series $g_{k}(n,z)$ defined in \eqref{partEisenstein} is valid for all $z\not\in \mathbb R$ and all $k$. It is clear from Lemma \ref{Lemma} that since the $f_k$ all vanish if $k$ is odd, then $g_k(n,\lambda)=0$ as well. It is also clear that the evenness of the $f_{2k}$ in the $z$-aspect yields evenness of $g_{2k}(n,z)$ in $z$. Observing that the partitions of $n$ represent identically the same set as the conjugate partitions of $n$, it follows from \eqref{partEisenstein} together with Lemma \ref{Lemma} that
\begin{equation*}%\label{partEisenstein3}
g_{2k}(n, -1/z)=z^{2k}\sum_{\lambda \vdash n} f_{2k}(\overline{\lambda}, z)=z^{2k}\sum_{\lambda \vdash n} f_{2k}(\lambda, z)=z^{2k}g_{2k}(n, z),
\end{equation*}
which completes the proof of the theorem.      \end{proof}

        \begin{proof}[Proof of Corollary \ref{Cor1}]
            Note from the geometric visualization of the Ferrers-Young lattice in $\mathbb C$, that for any $\lambda$ we have $|1+z| \leq |\omega|<|\lambda|\cdot |1+z|$ for every $\omega\in\mathscr F(\lambda,z) $. Then for $n=|\lambda|, \  z\not\in \{-1,0\}$, when $k\geq 0$ we have that $|f_{k}(\lambda,z)| \leq n  |1+z|^{-k}$, thus $|g_{k}(n,z)|\leq n\cdot p(n)|1+z|^{-k}$ with $p(n)$ the classical partition function (number of partitions of $n$). Therefore, for $k\geq0$ the convergence of $\mathscr G_{k}(z,q)$ as defined in \eqref{partEisenstein2} follows from the well-known convergence of the series \begin{equation*}\sum_{n\geq 1} n \cdot p(n)q^n\end{equation*} for $|q|<1$ (see \cite{Andrews}). By a similar argument, for $k<0$ let $k':=-k>0$; then we have that $|f_{k}(\lambda,z)| < n^{k'+1} |1+z|^{k'}$, thus $|g_{k}(n,z)|< n^{k'+1} p(n)|1+z|^{k'}$ and convergence follows by comparison with $\sum_{n\geq 1} n^{k'+1} p(n)q^n<\infty$ for $|q|<1$.
             %To see why this is the case, recall that for $|q|<1$ the function $\prod (1-q^n)^{-1}$ can be shown to be an analytic function in the unit disk $|q|<1$, and that the coefficient of $q^n$ is precisely $p(n)$. Thus the series in question is obtained by taking the derivative of this series and multiplying by $q$. A well known fact from complex analysis states that the derivative of an analytic function is again analytic, and so the above series converges for $|q|<1$.

By Theorem \ref{Thm1}, that the $g_{k}$ vanish when $k$ is odd yields the vanishing of $\mathscr G_{k}$; and the evenness of the $g_{2k}$ in the $z$-aspect induces evenness of $\mathscr G_{2k}(z)$ as well.

Finally, for the functional equation in $z$, it suffices to note by Theorem \ref{Thm1} that
\begin{flalign*}%\label{partEisenstein2.5}
\mathscr{G}_{2k}(-1/z)=  \sum_{n\geq 1} g_{2k}(n, -1/z)q^n  =  z^{2k}\sum_{n\geq 1} g_{2k}(n, z)q^n=z^{2k}\mathscr{G}_{2k}(z).
\end{flalign*}\end{proof}
%            Along similar lines, any number of further semi-modular forms seems possible. For instance, set $\rho:=e^{2\pi i z^{2k}}, z\in \mathbb H$. Recalling that $f_{2k}(\lambda, z)$ generalizes $|\lambda|$, then a generalization of the partition generating function $\sum_{\lambda \in \mathcal P}\rho^{|\lambda|}=\prod_{n\geq 1}(1-\rho^{n})^{-1}$ is the series
%           \begin{equation}
%            F_{4k}(z):=\sum_{\lambda\in \mathcal P}\rho^{f_{4k}(\lambda,z)},
%            \end{equation}
%             which can be seen to converge on the double half-plane $z\not\in \mathbb R$. The reader can confirm that $F_{4k}(-1/z)=F_{4k}(z)$ and $F_{4k}(-z)=F_{4k}(z)$.
%
\section{Further questions}
                Classical modular forms enjoy many nice interrelations, and structural properties such as being finitely generated (see \cite{Ono_web}), viz. any holomorphic integer-weight modular form can be expressed as a combination of the classical Eisenstein series $G_4(\tau), G_6(\tau)$. Our definition of semi-modular forms was algebraically inspired by ideas about generators of $GL_2(\mathbb Z)$, without looking in this study for analogues of growth conditions, being finitely-generated or other interesting properties of modular forms. %; we leave such questions for a future study.
                
                Then one wonders: Do the partition Eisenstein series form a basis for a suitably defined space of functions? %How would a different choice for the complementary matrix $R$ change the structure of the space? 
                Are there further connections to modular forms theory, such as links to quasi-modular forms or other generalized modular forms like Hilbert modular forms?
            Finally, considering the uniquely combinatorial-analytic constructions in Section \ref{Sect2whole}, might other families of semi-modular forms exist that involve ``complementary'' transformation matrices {\it different} from $R$ above, %, or %? Must they necessarily depend on partition symmetries, or can they arise
arising from other types of naturally-occurring mathematical structures --- either within the set of partitions or from elsewhere in mathematics? %If so, how would one construct them?

    \section*{Acknowledgments}

The authors are grateful to Agbolade Patrick Akande, Marie Jameson, David McGady, Ken Ono, Paul Pollack,  Larry Rolen, A. V. Sills and Ian Wagner for conversations that informed our study. We are also very thankful to the anonymous referee for comments that strengthened this work. The first author was partially supported by the Research and Training Group grant DMS-1344994 funded by the National Science Foundation.

\end{document}